\documentclass[11pt]{amsart}
\usepackage{amsmath, amssymb, amsthm, bm, mathtools, enumitem, caption}
\usepackage{hyperref}

\usepackage[noabbrev,capitalize]{cleveref}
\usepackage{adjustbox}
\crefname{equation}{}{}
\usepackage{ytableau}
\usepackage{graphicx, tikz}
\usepackage[textheight=8.95in, textwidth=7in]{geometry}

\usepackage{microtype}

\usepackage{graphicx}
\usepackage{amsopn,amsfonts,stmaryrd}
\usepackage{verbatim}
\usepackage{amsthm}
\usepackage{mathtools}
\usepackage{color}
\usepackage{enumitem}
\usepackage[framemethod=TikZ]{mdframed}
\usepackage{bbm}
\usepackage{mathrsfs}
\usepackage{booktabs}
\usepackage{caption}
\usepackage{bm}
\usepackage{tensor}
\usepackage{cancel}

\usepackage{mathrsfs}
\usepackage[scr=boondox]{mathalfa}

\newtheorem{theorem}{Theorem}[section]
\newtheorem{lemma}[theorem]{Lemma}
\newtheorem{corollary}[theorem]{Corollary}

\newtheorem*{conjecture*}{Conjecture}

\theoremstyle{definition}

\theoremstyle{remark}
\newtheorem*{remark}{Remark}
\newtheorem*{remarks}{Remarks}
\newtheorem{example}[theorem]{Example}

\numberwithin{equation}{section}

\newcommand{\N}{\mathbb N}

\newcommand{\wt}[1]{\widetilde{#1}}

\newcommand{\z}{\zeta}

\newcommand{\ol}{\overline}

\newcommand{\Pmod}{\hspace{-0.25cm}\pmod}

\newcommand{\SL}{\mathrm{SL}}

\newcommand{\Z}{\mathbb Z}

\newcommand{\IC}{{\mathbb C}}%Complex
\newcommand{\IZ}{{\mathbb Z}}%Integers
\newcommand{\IN}{{\mathbb N}}%Natural numbers
\newcommand{\IQ}{{\mathbb Q}}%Rationals
\newcommand{\IH}{{\mathbb H}}%quaternions
\newcommand{\QA}{\IQ\langle A\rangle}
\newcommand{\qsh}{\ast_\diamond}

\newcommand{\lrb}[1]{\left( #1 \right)}

\newcommand{\re}{{\rm Re}}
\newcommand{\im}{{\rm Im}}

\def\a{\alpha}

\def\d{\delta}

\def\l{\lambda}

\def\z{\zeta}

\def\s{\sigma}
\def\g{\gamma}
\def\t{\tau}

\newcommand{\CC}{\mathscr{C}}

\newcommand{\andd}{\quad \mbox{ and } \quad}
\newcommand{\where}{\quad \mbox{ where }}

\newcommand{\St}[1]{\left[ \begin{matrix} #1 \end{matrix} \right]}
\newcommand{\st}[1]{\left\{ \begin{matrix} #1 \end{matrix} \right\}}

\def\lp{\left(}
\def\rp{\right)}
\def\lb{\left[}
\def\rb{\right]}

\newcommand{\pmat}[1]{\left( \smallmatrix #1 \endsmallmatrix \right)}
\newcommand{\mat}[1]{\left( \begin{matrix} #1 \end{matrix} \right)}

\title[Quasimodularity and Limiting Behavior for Variations of MacMahon Series]{Quasimodularity and Limiting Behavior for Variations of MacMahon Series}
\date{\today}
\author{Caner Nazaroglu, Badri Vishal Pandey \and Ajit Singh}
\address{University of Cologne, Department of Mathematics and Computer Science,
	Weyertal 86-90,
	50931 Cologne, Germany}
\email{cnazarog@uni-koeln.de}
\email{bpandey@uni-koeln.de, badrivishal9451@gmail.com}
\address{Dept. of Mathematics \& Computing, Indian Institute of Technology (Indian school of Mines) Dhanbad, Jharkhand, India.}
\email{ajit94@iitism.ac.in}

\begin{document}
\begin{abstract}
Motivated by the 1920's seminal work of Major MacMahon, Amdeberhan--Andrews--Tauraso recently introduced an infinite family of $q$-series 
	\begin{align*}
		\mathcal{U}_{t}(a;q):= \sum_{1\le n_1<n_2<\cdots<n_t} \frac{q^{n_1+n_2+\cdots+n_t}}{(1+aq^{n_1}+q^{2n_1})(1+aq^{n_2}+q^{2n_2})\cdots (1+aq^{n_t}+q^{2n_t})}
	\end{align*}
and proved that these functions are linear combinations of quasimodular forms. In this paper, we study a broader family of $q$-series that contains the collection $\{\mathcal{U}_t\}_{t \in \IN}$. Using the theory of quasi shuffle algebras, we show that this extended family also lies in the algebra of quasimodular forms. Moreover, we determine the precise weights and levels of these functions, thereby making Amdeberhan--Andrews--Tauraso's result sharp.   We further investigate the limiting behavior of these functions. In particular, we demonstrate that the sequence of quasimodular forms~$\{\mathcal{U}_t(1;q)\}_{t\in\N}$ gives an approximation for the ordinary partition function. We also establish  infinitely many closed formulas for reciprocals of certain infinite products in terms of~$\mathcal{U}_{t}(a;q)$.
\end{abstract}

\maketitle
\section{Introduction and Statement of Results}
In exploring the relationship between integer partitions and divisor sums, 
MacMahon \cite{MacMahon} introduced an infinite family of $q$-series defined by
\begin{align}\label{eq:MacMahon}
	\mathcal{U}_{t}(q):=\sum_{n=0}^\infty\mathcal{M}(t;n)q^n= \sum_{1\le n_1<n_2<\cdots<n_t} \frac{q^{n_1+n_2+\cdots+n_t}}{(1-q^{n_1})^2(1-q^{n_2})^2\cdots (1-q^{n_t})^2},
\end{align}
where $\mathcal{M}(t;n)$ counts the sum of the products of the part multiplicities for partitions of $n$ with $t$ distinct part sizes. More precisely, recall that a partition $\lambda=(\lambda_1,\lambda_2,\ldots,\lambda_k)$ of $n$, denoted $\lambda\vdash n$, is a non-increasing sequence of positive integers that sum to $n$.
Considering the subset $\mathcal{P}_t (n)$ of partitions $\l$ of $n$ whose parts  assume exactly $t$ values, namely \smash{$1\leq n_{\l,1}<n_{\l,2}<\ldots<n_{\l,t}$}, with multiplicities \smash{$m_{\l,1},\ldots,m_{\l,t} \in \IN$} (so that \smash{$m_{\l,1} n_{\l,1} + \cdots + m_{\l,t} n_{\l,t} = n$}), we have  \smash{$\mathcal{M}(t;n)=\sum_{\lambda \in \mathcal{P}_t (n)} m_{\l,1} m_{\l,2} \cdots m_{\l,t}$}.
This partition generating function can also be recognized as a $q$-multiple zeta value and accordingly has links to fields such as enumerative geometry, representation theory, and topological string theory (see e.g.~\cite{BK16, HPHP15, Kreimer, Okounkov}). 
It further has ties to the theory of modular forms, since $\mathcal{U}_{t}(q)$ can be expressed as a linear combination of quasimodular forms on $\mathrm{SL}_2(\mathbb{Z})$ of weights up to $2t$ as shown by Andrews and Rose \cite{Andrews-Rose, rose}.

In recent years, research has increasingly examined questions of modularity and other related properties for various extensions of $\mathcal{U}_{t}(q)$ (see e.g.~\cite{AOS, AAT1,Bac24,Bringmann,  KMS}). The starting point of our paper is the related infinite family of $q$-series introduced very recently by Amdeberhan, Andrews, and Tauraso \cite{AAT2} as
\begin{align}\label{eq:AAT}
	\mathcal{U}_{t}(a;q):= \sum_{1\le n_1<n_2<\cdots<n_t} \frac{q^{n_1+n_2+\cdots+n_t}}{(1+aq^{n_1}+q^{2n_1})(1+aq^{n_2}+q^{2n_2})\cdots (1+aq^{n_t}+q^{2n_t})}.
\end{align}
This $q$-series includes many interesting objects as special cases. 
For example, for $t=1$, we have
\begin{align*}
	\mathcal{U}_1(a;q)&=\sum_{n\ge 1} \frac{q^n}{1+aq^n+q^{2n}}.
\end{align*}
If $a=0$, this $q$-series counts the representations of a number as the sum of two squares since
\begin{align*}
	\mathcal{U}_1(0;q) &= \sum_{n\ge 1} \frac{q^n}{1+q^{2n}} =
	\frac{\lrb{\sum_{n\in\Z} q^{n^2} }^2-1}{4}.
\end{align*}
When $a=1$, on the other hand, we have
\begin{align*}
	\mathcal{U}_1(1;q)&=\sum_{n\ge 1} \frac{q^n}{1+q^n+q^{2n}} = \frac{\sum_{n\in\Z} (-1)^n \, n \, q^{\frac{n(3n+1)}{2}}}{\sum_{n\in\Z} (-1)^n \, q^{\frac{n(3n+1)}{2}}},
\end{align*}
giving the first order Taylor coefficient with respect to the elliptic variable for the logarithm of a unary theta function.

We note that when $a=-2$, we get back MacMahon's original function $\mathcal{U}_t(q)$. Amdeberhan, Andrews, and Tauraso proved that $\mathcal{U}_{t}(a;q)$'s are mixed weight quasimodular forms for a congruence subgroup when $a=0,\pm1,2$. They further remarked \cite[Remark~8.1]{AAT2} that ``\emph{$\ldots$ we believe that each function $\mathcal{U}_{t}(a;q)$ $\mathrm{[}$for $a \in \{0,\pm 1\}$ $\!\! \mathrm{]}$ belongs to the space of quasimodular forms of weight $\le t$ for the congruence subgroup $\Gamma_0(24)$}". Here, as our first result, we prove that this indeed is the case. More precisely, for $t,k,r \in \IN$ we define
\begin{align}\label{eq:U-t-k-r}
	\mathcal{U}_{t,k,r}(a;q):= \sum_{1\le n_1<n_2<\cdots<n_t} \frac{q^{r(n_1+n_2+\cdots+n_t)}}{(1+aq^{n_1}+q^{2n_1})^k\cdots (1+aq^{n_t}+q^{2n_t})^k}=:\sum_{n=0}^\infty \mathcal{M}_{t,k,r}(a;n)q^n.
\end{align}
For convenience, we further define the shorthand
\begin{align}\label{eq:U-k-r}
	\mathcal{U}_{k,r}(a;q):=\mathcal{U}_{1,k,r}(a;q)=\sum_{n\ge 1} \frac{q^{rn}}{(1+aq^n+q^{2n})^k}.
\end{align}
%{\color{red}and by a (quasi)modular form of level $N$, we mean a (quasi)modular form for the congruence subgroup $\Gamma_0(N)$.}
\begin{theorem}\label{thm:weight-and-level-of-U}
	Assuming the notations above, we have the following:\footnote{Throughout, by a (quasi)modular form of level $N$, we mean a (quasi)modular form for the congruence subgroup~$\Gamma_0(N)$.}
	\begin{enumerate}[leftmargin=*]
		\item[(1)] $\mathcal{U}_{k,k}(a;q)$ is a mixed weight quasimodular form of highest weight $k$ and level $4, 3, 6$ for $a=0, 1, -1$, respectively, and highest weight $2k$ and level $2$ for $a=2$. Moreover, $\mathcal{U}_{k,k}(0;q)$ is modular when $k$ is odd.
		\item[(2)] $\mathcal{U}_{t,k,k}(a;q)$ is an isobaric polynomial of degree $t$ in the variables $\mathcal{U}_{k,k}(a;q), \mathcal{U}_{2k,2k}(a;q), \ldots, \mathcal{U}_{tk,tk}(a;q)$.\footnote{
An isobaric polynomial of degree $t$ in variables $X_1,\ldots,X_t$ has monomials $X_1^{a_1}X_2^{a_2}\cdots X_t^{a_t}$ satisfying 
$a_1+2a_2+\cdots +ta_t=t$.	
}
In particular, it is a mixed weight quasimodular form with highest weight $tk$ and level $4, 3, 6$ for $a=0, 1, -1$, respectively, and highest weight $2tk$ and level $2$ for $a=2$. 
	\end{enumerate}
\end{theorem}
\begin{example}\label{eg:1}
	We have
	\begin{align*}
		\mathcal{U}_{2,k,k}(a;q) &= \frac{1}{2}\mathcal{U}_{k,k}^2(a;q)-\frac{1}{2}\mathcal{U}_{2k,2k}(a;q), \\
		\mathcal{U}_{2,2}(0;q) &= G_{2}(2\tau)-4G_{2}(4\tau) - \frac{1}{8}, \\
		\mathcal{U}_{2,2}(2;q) &= -\frac{G_4(\tau )}{6}+\frac{8}{3} G_4(2 \tau )+\frac{G_2(\tau )}{6}-\frac{2}{3} G_2(2 \tau )-\frac{1}{32}, \\
		\mathcal{U}_{2,2}(1;q) &= -3G_2(3\tau)+\frac{G_2(\tau)}{3}-\frac{G_1(\chi_{3,2};\tau)}{3}-\frac{1}{18} ,\\
		 \mathcal{U}_{2,2}(-1;q) &= -12G_2(6\tau)+3G_2(3\tau)+\frac{4}{3}G_2(2\tau)-\frac{1}{3}G_2(\tau)+\frac{2}{3}G_1\lrb{\chi_{3,2};2\tau} + \frac{1}{3}G_1\lrb{\chi_{3,2};\tau}-\frac{1}{2},
	\end{align*}
	where $G_k(\tau)$ and $G_k(\chi;\tau)$ are Eisenstein series defined in Section~\ref{sec:prelim} and $\chi_{3,2}$ denotes the unique primitive Dirichlet character modulo $3$.
\end{example}

As a direct corollary we make the result of Amdeberhan, Andrews, and Tauraso \cite[Theorem~8.3]{AAT2} sharp by determining the precise levels of the quasimodular forms in consideration. More precisely, choosing $k=1$ in Theorem~\ref{thm:weight-and-level-of-U}(2) gives us the following corollary.
\begin{corollary}\label{Cor:weight-and-level-of-U}
The function $\mathcal{U}_t(a;q)$ is a linear combination of quasimodular forms of weight $\leq t$ on the group $\Gamma_0 (4)$ for $a=0$, $\Gamma_0 (3)$ for $a=1$, and  $\Gamma_0 (6)$ for $a=-1$.
\end{corollary}
\begin{remarks}\hfill\break
(1) We note that as a special case we also recover \cite[Corollary~6.1]{AAT2}, which says that $\mathcal{U}_t(2;q)$ is a linear combination of quasimodular forms of weight $\leq 2t$ on the group $\Gamma_0 (2)$.\\
(2) Our results in Theorem~\ref{thm:weight-and-level-of-U} examine the case $r=k$, but our considerations have natural generalizations to the case $r \neq k$ as well. For example, in Theorems \ref{thm:Qk_quasipolynomial_a_1} and \ref{thm:Qk_quasipolynomial_a_m1}, we find that for $a = \pm 1$, the symmetric combinations $\mathcal{U}_{k,r}(a;q)+\mathcal{U}_{k,2k-r}(a;q)$ with $r \in \{1,\ldots,k\}$ also decompose to (quasi)modular forms of weight~$\leq k$. Furthermore, the methods we employ there also makes it clear that odd Eisenstein series and their higher level analogues (possibly with even weight) appear instead in the decomposition of the antisymmetric combination $\mathcal{U}_{k,r}(a;q)-\mathcal{U}_{k,2k-r}(a;q)$. It would be interesting to have a deeper examination of the appearance and applications of holomorphic quantum modularity in this context.\\
(3) The generating function of the MacMahon $q$-series $\mathcal{U}_{t}(-2;q) = \mathcal{U}_t(q)$ is elegantly related to Eisenstein series via the following identity \cite[Theorem~1.1]{Bac24}:
\begin{equation*}
1 + \sum_{t\geq 1} \mathcal{U}_{t} (q)  X^{2t}
=
\frac{2}{X} \arcsin\!\lp \frac{X}{2} \rp 
\exp \!\lp 2 \sum_{j \geq 1} \frac{(-1)^{j-1}}{(2j)!}
G_{2j} (\t) \lp 2 \arcsin \!\lp \frac{2}{X} \rp \rp^{2j}
\rp .
\end{equation*}
When the sum defining $\mathcal{U}_t$ is restricted through congruence conditions, one obtains similar identities involving generalized Eisenstein series with characters (see \cite[Theorem~3.4]{KMS}).
This generalization encompasses, for instance, the series $\mathcal{U}_{t}(2;q)$ as a special case.

Using the results and discussion of Theorem \ref{thm:weight-and-level-of-U} and Lemma \ref{lem:QS-U} in Section \ref{sec:quasimod_Utkk}, we computationally observe an analogous structure for  $\mathcal{U}_{t}(a;q)$ with $a \in \{0, \pm1\}$.
More specifically, our computations suggest the existence of power series $f_0(a;X)$ and $f(a;X)$ such that
\begin{equation*}
1+ \sum_{t\geq 1} \mathcal{U}_{t}(a;q)  X^{t} =  f_0(a;X) 
\exp\!\left(\sum_{j \geq 1} \frac{(-1)^{\lfloor\frac{j-1}{2}\rfloor}}{j!} \mathbb{G}_{j}(a;\t) f(a;X)^{j} \right),
\end{equation*}
where the functions $\mathbb{G}_{j}(a;\t)$ are defined as follows:
\begin{align*}
\mathbb{G}_{j}(1;\t) &:=
\begin{cases}
3^j G_j(3\tau) - G_j(\tau)
\quad &\mbox{if } 2 \mid j, \\
\sqrt{3} \, G_j (\chi_{3,2};\t)
\quad &\mbox{if } 2 \nmid j,
\end{cases}
\\
\mathbb{G}_{j}(0;\t) &:=
\begin{cases}
4^j G_j(4\tau) - 2^j G_j(2\tau)
\quad &\mbox{if } 2 \mid j, \\
2 \, G_j(\chi_{4,2};\tau)
\quad &\mbox{if } 2 \nmid j,
\end{cases}
\\
\mathbb{G}_{j}(-1;\t) &:=
\begin{cases}
6^j G_j(6\tau) - 3^j G_j(3\tau) - 2^j G_j(2\tau) + G_j(\tau)
\quad &\mbox{if } 2 \mid j, \\
\sqrt{3} \lp 2^j G_j(\chi_{3,2};2\tau) + G_j(\chi_{3,2};\tau) \rp
\quad &\mbox{if } 2 \nmid j.
\end{cases}
\end{align*}

\noindent(4) By Theorem \ref{thm:weight-and-level-of-U} and the fact that $G_2(\tau)$ is a modular form modulo any prime $p$, the functions $\mathcal{U}_{t,k,k}(a;q)$ are modular forms modulo $p$. Hence, using Serre's theory of $p$-adic modular forms \cite{Ser76}, there are infinitely many non-nested arithmetic progressions (depending on $p$) such that
the coefficients of $\mathcal{U}_{t,k,k}(a;q)$ on these
arithmetic progressions satisfy congruences modulo $p$. Computational experiments suggest the following: 
\begin{align*}
&\mathcal{M}_{3m+1\pm 1,2,2}(1;3n+2)\equiv 0\Pmod 3,\quad \mathcal{M}_{3m+1\pm 1,2,2}(-2;3n+2)\equiv 0\Pmod 3,\quad m\in \N,\\
&\mathcal{M}_{5 ,2,2}(-2;3n+2)\equiv 0\Pmod 3,\quad\mathcal{M}_{2,1,1}(\pm 1;4n+1)\equiv 0\Pmod 4,\quad\mathcal{M}_{2,1,1}( 1;8n+5)\equiv 0\Pmod 8.
\end{align*}
Computational data also suggest congruences for $k\neq r$, where $\mathcal{U}_{t,k,r}(a;q)$ is not quasimodular. For example,
 $$\mathcal{M}_{1,3,1}(-2;9n+4)\equiv 0\Pmod 3, \quad \mathcal{M}_{1,3,1}( -2;7n+2)\equiv 0\Pmod 7,\quad \mathcal{M}_{1,3,1}( -2;8n+4)\equiv 0\Pmod 7.$$
This makes the study of congruences for $\mathcal{U}_{t,k,r}(a;q)$ an interesting test case for examining congruences within more novel forms of modularity.
\end{remarks}
Another interesting direction to pursue for the functions we study here is their limiting behavior in $t$.
Recently in \cite{AOS}, Amdeberhan, Ono, and the third author studied this question for the family of quasimodular forms $\mathcal{U}_t(q)$. They showed that the sequence $\{\mathcal{U}_t(q)\}_{t\in\N}$ converges to the generating function of the three colored partitions, i.e.  if $t \in \IN$, then one has\footnote{
Here $(a;q)_\infty:=\prod_{n=0}^\infty(1-aq^n)$ is the $q$-Pochhammer symbol for any two complex numbers $a,q$ with $|q|<1$.
}
\begin{align}\label{AOS_limit}
q^{-\frac{t(t+1)}{2}} \, \mathcal{U}_t(q)=\frac{1}{(q;q)_\infty^3}+O(q^{t+1}).
\end{align}
Moreover, Ono\footnote{The question was raised in a research talk in the ``Partition theory, $q$-series and related Topics" seminar at Michigan Tech.} suggested that there should exist natural families of quasimodular forms that, in the limit, give $k$-colored partition generating function.
This question was resolved by Bringmann, Craig, van Ittersum, and the second author in  \cite{Bringmann}.
For each $k$, they gave two families of MacMahon-like $q$-series which give the $k$-colored partition generating function in the limit. They further showed that one family is quasimodular whereas the other is mock modular with an explicit, real analytic modular completion.
Correspondingly, it is natural to ask whether the generalization of the MacMahon function here, also quasimodular in special cases, has an interesting limiting behavior. As our next result, we show that this is indeed the case.
\begin{theorem}\label{thm:lim-behaviour}
	For $t,k,r\in\N$, we have
	\begin{align*}
		q^{-\frac{rt(t+1)}{2}}\mathcal{U}_{t,k,r}(a;q)=\prod_{n\geq 1}\frac{1}{(1-q^{nr})(1+aq^n+q^{2n})^{k}}+O(q^{t+1}).
	\end{align*}
\end{theorem}
This result is established in Section \ref{sec:limiting_Utkr} by finding the limiting behavior for a more general family of functions that contains the functions $\mathcal{U}_{t,k,r}(a;q)$ as special instances.
We should also remark that the above theorem establishes a connection between $\mathcal{U}_t(a;q)$ and many interesting partition functions. For example, when $a=-2$, we recover the result \eqref{AOS_limit}.  The cases $a=0,\pm1,2$, on the other hand, yield the limiting functions 
\begin{align*}
\prod_{n\geq 1}\frac{1}{(1-q^n)(1+q^{2n})} &:=\sum_{n\geq 0} a(n)q^n=1 + q + q^2 + 2q^3 + 3q^4 + 4q^5 + 5q^6 + 7q^7 + 10q^8 + \cdots,\\
\prod_{n\geq 1}\frac{1}{(1-q^{3n})} &:=\sum_{n\geq 0} b(n)q^n=1 + q^3 + 2q^6 + 3q^9 + 5q^{12} + 7q^{15} + 11q^{18} + 15q^{21} + 22q^{24}+ \cdots,\\
\prod_{n\geq 1}\frac{(1+q^{n})}{(1+q^{3n})(1-q^{n})} &:=\sum_{n\geq 0} c(n)q^n=1 + 2q + 4q^2 + 7q^3 + 12q^4 + 20q^5 + 32q^6 + 50q^7 +\cdots,\\
\prod_{n\geq 1}\frac{(1-q^n)}{(1-q^{2n})^2} &:=\sum_{n\geq 0} d(n)q^n=1 - q + q^2 - 2q^3 + 3q^4 - 4q^5 + 5q^6 - 7q^7 + 10q^8 - 13q^9 + \cdots.
\end{align*}
Here we note that $a(n)$ counts the number of partitions of $n$ in which all odd parts are distinct and there is no restriction on the even parts, whereas $b(n)$ counts the number of partitions of $n$ into parts multiple of $3$. The numbers $c(n)$, on the other hand, count the number of partitions of $2n$, in which both odd parts and parts that are multiples of 3 occur with even multiplicities with no restriction on the remaining even parts. Finally, we note that the generating function of $d(n)$ is the
reciprocal of $\psi(q)$, Ramanujan's theta function \cite{Ram}, and $d(n)$'s are equal to $a(n)$ up to sign. Theorem~\ref{thm:lim-behaviour} then produces approximate generating functions for all of these partition-related numbers in terms of $\mathcal{U}_t (a;q)$.

There is also a great deal of information one can extract by focusing on the combinatorial aspects of the objects here. For example, in a recent paper, 
Ono and the third author \cite{OS} studied MacMahon's series from this lens and showed that $\mathcal{U}_t(q)$ satisfy infinitely many identities, which illustrate the ubiquity of the $\mathcal{M}(t; n)$ partition functions. More precisely, they proved that for any $t \in \IN$, we have
\begin{align}\label{eq:three-color-partition}
	\frac{1}{(q;q)_{\infty}^3}=q^{-\frac{t(t+1)}{2}}\sum_{m=t}^\infty\binom{2m+1}{m+t+1}\mathcal{U}_m(q),
\end{align}
which extends the limiting behavior in equation \eqref{AOS_limit} to an exact identity.
A recent preprint by Jin and two of the authors \cite{JPS} obtain further such results relating MacMahon and MacMahon-type $q$-series with Rogers-Ramanujan identities, the theta series associated to the non-trivial character modulo $6$, and infinite families of restricted partition functions. 
In view of the above discussion, it is natural to ask whether Theorem~\ref{thm:lim-behaviour} hints similar explicit identities for each $t \in \IN$. We demonstrate that this is indeed the case, where the term \smash{$q^{-\frac{1}{2} t(t+1)}\, \mathcal{U}_{t}(a;q)$} from the limiting formula represents the first term in a closed formula involving $\mathcal{U}_{t+1}(a;q)$, $\mathcal{U}_{t+2}(a;q)$, $\ldots$. 
\begin{theorem}\label{Explicit_Identities}
For any $t \in \IN$ we have 
\begin{equation*}
\prod_{n\geq 1}\frac{1}{(1-q^n)(1+aq^n+q^{2n})}=q^{-\frac{t(t+1)}{2}}
\sum_{m=t}^\infty \mathcal{U}_m(a;q)
\sum_{\g = 0}^{m-t} \binom{m}{\g} \binom{m-\g}{\lfloor \frac{m-\g-t}{2} \rfloor}
(-a)^{\g} .
\end{equation*}
\end{theorem}
\begin{remark}
	Substituting \( a = -2 \) into Theorem~\ref{Explicit_Identities} and comparing it with \eqref{eq:three-color-partition} (recalling that \( \mathcal{U}_m(q) = \mathcal{U}_m(-2;q) \)) immediately yields the identity (for \( t, m \in \mathbb{N} \) with \( m \ge t \)):
	\begin{align*}
		\binom{2m+1}{m+t+1} = \sum_{\g = 0}^{m-t} \binom{m}{\g} \binom{m-\g}{\lfloor \frac{m-\g-t}{2} \rfloor} 2^{\g}.
	\end{align*}
	This identity can also be derived directly by expanding \( (1+z)(z+z^{-1}+2)^m \) in two different ways.
\end{remark}

The organization of the paper is as follows. In Section \ref{sec:prelim}, we recall some preliminary notions and results on modular and quasimodular forms, Eulerian polynomials, and quasi-shuffle algebras.
In Section \ref{sec:quasimod_Utkk}, we employ these ideas to prove Theorem \ref{thm:weight-and-level-of-U}. In particular, we use quasi-shuffle algebras in order to obtain the decomposition of $U_{t,k,k}(a;q)$ into $U_{k,k}(a;q), U_{2k,2k}(a;q), \ldots, U_{tk,tk}(a;q)$. Then we prove the modularity properties of $U_{k,k}(a;q)$ for $a=0,\pm 1,2$ by finding an explicit decomposition to Eisenstein series. Our approach to the $a = \pm 1$ case emphasizes the utility of Eulerian polynomials and Stirling numbers in establishing such decompositions.
Next in Section \ref{sec:limiting_Utkr}, we examine the limiting behavior of $\mathcal{U}_{t,k,r}(a;q)$ in $t$ and prove Theorem \ref{thm:lim-behaviour}.
Finally, in Section \ref{sec:explicit_identity}, we prove the explicit identities from Theorem \ref{Explicit_Identities}.

%%%%%%%%%%%%%%%%%%%%%%%%%%%%%%%%%%%%%%%%%%%%%
\section*{Acknowledgments}
The first author is supported by the SFB/TRR 191 “Symplectic Structure in Geometry, Algebra and Dynamics”, funded by the DFG (Projektnummer 281071066 TRR 191).
The second author is funded by the European Research Council (ERC) under the European Union’s Horizon 2020 research and innovation programme (grant agreement No. 101001179). The third author is grateful for the support of a Fulbright Nehru Postdoctoral Fellowship at the University of Virginia, USA, and is currently supported by the INSPIRE Faculty Fellowship at IIT (ISM) Dhanbad, India. The authors would like to thank Jan-Willem van Ittersum for enlightening discussions on quasi-shuffle algebras and Toshiki Matsusaka for bringing to our attention the utility of MacMahon-type $q$-series as generating functions of Eisenstein series. Finally, the authors thank the anonymous referees for their valuable comments and suggestions, which improved both the quality and the exposition of the paper.

%%%%%%%%%%%%%%%%%%%%%%%%%%%%%%%%%%%%%%%%%%%%%
\section{Preliminaries}\label{sec:prelim}

\subsection{Eisenstein series on congruence subgroups and quasimodular forms}
We start our discussion with a quick reminder on a family of Eisenstein series on congruence subgroups, where further details can be found in \cite{CS2017}.
Let $k \in \IZ$, $N \in \IN$ and assume that $\chi$ is a Dirichlet character modulo $N$ satisfying $\chi(-1)=(-1)^k$.
Then for $\t \in \IH$ (with $\IH$ denoting the complex upper half-plane) and $s \in \IC$ with $k+2\re(s)>0,$ we define the nonholomorphic Eisenstein series
\begin{align}\label{eq:Eisenstein-series-s}
\mathfrak{G}_k(\chi;\tau;s) &:= \frac{1}{2}
\sum_{\substack{(c,d)\in\IZ^2 \setminus \{(0,0)\} \\ N \mid c}} 
\frac{\ol{\chi(d)}}{(c\tau+d)^k}
\frac{\im (\t)^s}{|c\tau+d|^{2s}} .
\end{align}
For $k \geq 3$, we can set $s=0$ and obtain a holomorphic function of $\t$ that (up to an overall normalization) we denote by $G_k (\chi; \t)$. The resulting function is in $M_k (\Gamma_0 (N), \chi)$, the space of holomorphic modular forms of weight $k$ and character $\chi$ on $\Gamma_0 (N)$, which is the congruence subgroup of $\mathrm{SL}_2 (\IZ)$ defined by
\begin{equation*}
\Gamma_0 (N) := \left\{ \mat{a & b \\ c & d} \in \mathrm{SL}_2 (\IZ) :\ c \equiv 0 \, (\mathrm{mod}\,N) \right\}.
\end{equation*}
More specifically, a holomorphic function $g : \IH \to \IC$ is in $M_k (\Gamma_0 (N), \chi)$ if it satisfies the transformation
\begin{equation*}
g\!\lp \frac{a \t+b}{c\t+d} \rp = \chi (d) \, (c\t+d)^k \, g(\t)
\quad \mbox{for all } \mat{a & b \\ c & d} \in \Gamma _0 (N)
\end{equation*}
and if \smash{$(c\t+d)^{-k} g\!\lp \frac{a \t+b}{c\t+d} \rp$} is bounded as $\im (\t) \to \infty$ for all $\pmat{a & b \\ c & d} \in \SL_2 (\IZ)$.
We summarize these facts in the following lemma. Note that we restrict to primitive characters since they lead to simple Fourier expansions and the others can be expressed in terms of these.

\begin{lemma}\label{lem:eis-with-char-kge3}
Let $k,N \in \IN$, $k \geq 3$, and $\chi$ be a primitive Dirichlet character modulo $N$ with $\chi(-1)=(-1)^k$. Then we have\footnote{Here and throughout $q := e^{2 \pi i \t}$.}
\begin{align*}
G_k(\chi;\tau) &:= \frac{L(\chi,1-k)}{2} + \sum_{n \geq 1} \sigma_{k-1}(\chi;n) \, q^n
\, \in M_k(\Gamma_0(N),\chi),
\end{align*}
where \smash{$\sigma_{k-1} (\chi;n) := \sum_{d|n} \chi(d) \, d^{k-1}$}.
\end{lemma}

In the special case of the trivial character (with $N=1$), we find the classical Eisenstein series, which, for $k \geq 4$ even, are in $M_k(\mathrm{SL}_2(\Z))$, the space of holomorphic modular forms on $\SL_2 (\IZ)$.
In this case, we simplify the notation from Lemma \ref{lem:eis-with-char-kge3} by omitting $\chi$ to write \smash{$\sigma_{k-1}(n):=\sum_{d|n} d^{k-1}$} and
\begin{equation}\label{eq:Eis-k-even}
G_k(\tau) := \frac{\z (1-k)}{2} + \sum_{r\geq 1} \sigma_{k-1}(r) q^{r}
= -\frac{B_k}{2k} + \sum_{m,r\geq 1} m^{k-1} q^{mr}.
 %\in M_k(\mathrm{SL}_2(\Z))\quad \mbox{for } k \geq 4 \mbox{ even}.
\end{equation}

These considerations can be extended to the cases $k=1$ and $k=2$ by noting that the non-holomorphic Eisenstein series $\mathfrak{G}_k(\chi;\tau;s)$ analytically continues in $s$ to
an entire function for $k \geq 1$.

\begin{lemma}\label{lem:eis-with-char-k=1,2}
	Let $\chi$ be a primitive Dirichlet character modulo $N \geq 2$. Then we have the following:
	\begin{enumerate}[leftmargin=*]
		\item[1)] If $\chi$ is an even character, that is $\chi(-1)=1$, then
		\begin{align*}
			G_2(\chi;\tau) := \frac{L(\chi,-1)}{2} + \sum_{n\ge1} \sigma_1(\chi;n) \, q^n \in M_2(\Gamma_0(N),\chi).
		\end{align*}
		\item[2)] If $\chi$ is an odd character, that is $\chi(-1)=-1$, then we have
		\begin{align*}
			G_1(\chi;\tau) := \frac{L(\chi,0)}{2} + \sum_{n\ge1} \sigma_0(\chi;n) \, q^n \in M_1(\Gamma_0(N),\chi).
		\end{align*}
	\end{enumerate}
\end{lemma}

When $\chi$ is trivial (so that $N=1$) and $k=2$, on the other hand, the analytic continuation to $s=0$ does not yield a holomorphic function of $\t$. Instead we find that
\begin{equation}\label{eq:G2_fourier_expansion}
G_2(\tau) := \frac{\z (-1)}{2} + \sum_{r\geq 1} \sigma (r) q^{r}
= -\frac{1}{24} + \sum_{m,r\geq 1} m \, q^{mr} \in \wt{M}_2(\mathrm{SL}_2(\Z)),
\end{equation}
where $\wt{M}_k(\mathrm{SL}_2(\Z))$ denotes the space of weight $k$ quasimodular forms on $\SL_2 (\IZ)$ (see \cite{BGRZ2008} for more details). In particular, the function $G_2(\tau)$ does not transform like a modular form but it instead satisfies
\begin{equation*}
G_2\!\left(\frac{a\tau+b}{c\tau+d}\right) = (c\tau+d)^2 \, G_2(\tau) 
+ \frac{i c (c\t+d)}{4\pi}
\quad \mbox{for all } \mat{a & b \\ c & d} \in \SL_2 (\IZ).
\end{equation*}
What instead transforms like a modular form is the non-holomorphic combination \smash{$G_2 (\t) + \frac{1}{8 \pi \im (\t)}$}. 
In general, we say that~$g: \IH \to \IC$ is a {\it quasimodular form of weight $k$, depth $h$, and with character $\chi$ on $\Gamma_0 (N)$} if there exist holomorphic functions $g_j:\mathbb{H}\to\mathbb{C}$ for $j\in\{0,\ldots,h\}$ with $g_h \neq 0$ such that\footnote{
Note that the given transformation law immediately implies that $g=g_0$ and that
(see e.g.~\cite{ChoieLee} or \cite{Royer})
\begin{equation*}
g_r\!\left(\frac{a\tau+b}{c\tau+d}\right) = \chi(d) (c\tau+d)^{k-2r} \sum_{j=0}^{h-r} \mat{r+j \\ r} 
g_{r+j} (\tau)\left(\frac{c}{c\tau+d}\right)^j \quad \mbox{for all }\begin{pmatrix}a&b\\c&d\end{pmatrix}\in\Gamma_0(N)
\mbox{ and } r \in \{0,\ldots,h\}.
\end{equation*}
}
\begin{align*}
g\!\left(\frac{a\tau+b}{c\tau+d}\right) = \chi(d) (c\tau+d)^{k} \sum_{j=0}^h g_j(\tau)\left(\frac{c}{c\tau+d}\right)^j \quad \mbox{for all }\begin{pmatrix}a&b\\c&d\end{pmatrix}\in\Gamma_0(N)
\end{align*}
and \smash{$(c\t+d)^{-(k-2r)} g_r\!\lp \frac{a \t+b}{c\t+d} \rp$} is bounded as $\im (\t) \to \infty$ for all $r \in \{0,\ldots,h\}$ and $\pmat{a & b \\ c & d} \in \SL_2 (\IZ)$.
The set of all such weight $k$ quasimodular forms together with the zero function constitutes a vector space, which we denote by $\wt{M}_k (\Gamma_0 (N), \chi)$.
Using the properties above and the transformation property of $G_2 (\t)$, one can easily deduce the structure theorem of quasimodular forms stating that a level $N$ quasimodular form as above is a polynomial in $G_2 (\t)$ with modular coefficients (for the character $\chi$ on $\Gamma_0 (N)$).

Finally, there are a number of natural operators defined on the space of 
(quasi)modular forms, such as Hecke operators, Atkin-Lehner operator, etc.
One such Atkin-Lehner operator is the operator $V_\ell$ for $\ell\in\N$, which acts on $f\in M_k(\Gamma_0(N),\chi)$ by
\begin{align}\label{eq:V-op}
	V_\ell(f) (\t) :=f(\ell\tau).
\end{align}
The function $V_\ell(f)$ is again a modular form and lies in $M_k(\Gamma_0(N\ell),\chi)$. An analogous conclusion also holds for quasimodular forms.

Finally, we find it convenient for our discussion to define
\begin{equation}\label{eq:Fk_definition}
F_k (\t) := \sum_{n \geq 1} \s_{k-1} (n) q^n =
\sum_{m,n\geq 1} m^{k-1} q^{mn}
\mbox{ and }
F_k (\chi; \t) := \sum_{n \geq 1} \s_{k-1} (\chi, n) q^n =
\sum_{m,n\geq 1} \chi(m) m^{k-1} q^{mn}
\end{equation}
for any $k \in \IN$ and Dirichlet character $\chi$. These are equal to $G_k (\t)$ and $G_k (\chi; \t)$, respectively, up to the addition of the constant Fourier coefficients (for appropriate $k$ and $\chi$ giving (quasi)modular forms).

\subsection{Eulerian polynomials and partial fractions} 
In parts of our proof of Theorem \ref{thm:weight-and-level-of-U}, we employ Eulerian polynomials to identify Eisenstein series (including those at higher levels) from the $q$-series under consideration. These polynomials were introduced by Euler through the relation
\begin{equation}\label{eq:Eulerian_polynomial_definition}
\sum_{m=1}^\infty m^{k-1} t^m = \frac{t \, P_{k-1} (t)}{(1-t)^k}
\quad \mbox{for } k \in \IN
\end{equation}
to study the zeta function at nonpositive integers as (with $\zeta (s,a)$ denoting the Hurwitz zeta function)\footnote{Here and throughout, we use $\z_N$ to denote $ e^{\frac{2\pi i}{N}}$.}${}^{,}$\footnote{
This follows from analytically continuing $\mathrm{Li}_s (\z_r^n) = r^{-s} \sum_{m=1}^r \z_r^{mn} \z (s,  \frac{m}{r})$ and noting $\mathrm{Li}_{1-\ell} (z) = \frac{z P_{\ell-1} (z)}{(1-z)^\ell}$ for $\ell \in \IN$.
}
\begin{equation}\label{eq:Eulerian_polynomial_zeta_relation}
\frac{\z_r^n P_{\ell - 1} \!\lp \z_r^n \rp}{\lp 1-\z_r^n \rp^\ell} = r^{\ell-1} \sum_{m=1}^r \z_r^{mn}
\z \!\lp 1-\ell, \frac{m}{r} \rp 
\quad \mbox{for }
\ell \in \IN, r \in \IN_{\geq 2}, \mbox{ and } n \in \{1,2,\ldots,r-1\}.
\end{equation}
The first few of these polynomials are given as
\begin{equation*}
P_0 (t) = 1, \qquad
P_1 (t) = 1, \qquad
P_2 (t) = 1+ t, \qquad
P_3 (t) = 1+4t+t^2.
\end{equation*}
For $n \geq 1$, the polynomial $P_n (t)$ has degree $n-1$ and it satisfies the symmetry property
\begin{equation}\label{eq:Eulerian_polynomial_symmetry}
P_n (t) = t^{n-1} P_n(1/t).
\end{equation}
Frobenius \cite{Frobenius} gave an explicit expansion of $P_n (t)$ around $t=1$ as
\begin{equation}\label{eq:Eulerian_polynomial_tm1_decomposition}
P_n (t) = \sum_{r=0}^n r! \st{n \\ r}  (t-1)^{n-r}
\quad \mbox{for } n \geq 0,
\end{equation}
where \smash{$\left\{ \begin{smallmatrix} n \\ r \end{smallmatrix} \right\}$} denotes a Stirling number of the second kind.
The significance of this identity is that it gives a way to convert between an expansion into fractions appearing on the right hand side of equation \eqref{eq:Eulerian_polynomial_definition}, which are natural for Eisenstein series, and a partial fraction expansion. To be more concrete, Frobenius' identity \eqref{eq:Eulerian_polynomial_tm1_decomposition} can be naturally rewritten as
\begin{equation*}
\frac{P_{k-1} (t)}{(1-t)^k} =
\sum_{r=1}^k (-1)^{k-r} (r-1)! \st{k-1 \\ r-1} \frac{1}{(1-t)^r} 
\quad \mbox{ for } k \geq 1.
\end{equation*}
Inversion of this relation is well-known in combinatorics (see e.g.~\cite{BS95}) and given in terms of 
unsigned Stirling numbers of the first kind (denoted by 
$\left[ \begin{smallmatrix} n \\ r \end{smallmatrix} \right]$) as
\begin{equation}\label{eq:eulerian_decomposition}
\frac{1}{(1-t)^r} = 
\frac{1}{(r-1)!} \sum_{k=1}^r \St{r-1 \\ k-1} \frac{P_{k-1} (t)}{(1-t)^k}
\quad \mbox{for } r \geq 1 .
\end{equation}
We note a few particular values for unsigned Stirling numbers of the first kind as
\begin{equation}\label{eq:unsigned_stirling_first_kind_particular_values}
\St{n \\ 0} = \d_{n,0} \ \  \mbox{for } n \geq 0, \qquad
\St{n \\ n} = 1 \ \  \mbox{for } n \geq 0
\andd
\St{n \\ 1} = \begin{cases}
0 \quad &\mbox{if } n = 0, \\
(n-1)! \quad &\mbox{if } n \geq 1.
\end{cases}
\end{equation}

\subsection{Quasi-shuffle algebras}\label{sec:quasi_shuffle_intro}
Another crucial ingredient for the proof of Theorem \ref{thm:weight-and-level-of-U} is the notion of quasi-shuffle algebras \cite{Hoffman, HI17, IKOO11}. To describe this, let $A$ be a countable set, called the \textit{set of letters}. We endow the $\IQ$ vector space $\IQ A$ with a bilinear, commutative, and associative product $\diamond$. Then we consider the set of words $\langle A \rangle$, which consists of \textit{words} $a_1 \ldots a_\ell$ built from letters $a_1, \ldots,a_\ell \in A$ for $\ell \in \IN_{0}$. Here we denote the empty word (for $\ell =0$) by $1$.
The \textit{quasi-shuffle product} $\qsh$ is a bilinear, associative, and commutative product on the vector space $\QA$ given by linearly extending the product on words  defined through
\begin{equation}\label{eq:quasi_shuffle_product}
1 \qsh w = w \qsh 1 = w
\ \mbox{ and }\ 
a_1 w_1 \qsh a_2 w_2 
= a_1 (w_1 \qsh a_2 w_2) + a_2 (a_1 w_1 \qsh w_2) + (a_1 \diamond a_2) (w_1 \qsh  w_2)\end{equation}
for letters $a_1,a_2 \in A$ and words $w,w_1,w_2 \in \langle A \rangle$.

Quasi-shuffle algebras have been an effective tool in treating iterative sums like \eqref{eq:MacMahon} and \eqref{eq:AAT} and hence studying multiple zeta values (MZV) and their variations like $q$-MZV, multiple polylogarithms, multiple harmonic sums etc. (see e.g.~\cite{BK20, B05,Z15, H18}). Furthermore, they have found broad applications across various areas of mathematics, notably in the theory of Hopf algebras \cite{H18}, and in partition theory through the use of $q$-brackets~\cite{BI22}.
To understand the relevance of the concept to our consideration, let us consider the algebra of formal power series $\IQ[[q]]$ and the subalgebra $\mathcal{H}_{a,q}$ generated (and spanned) by 
\begin{equation*}
h_{k_1,\ldots,k_t;r_1,\ldots,r_t} (a;q) := 
\sum_{1\leq n_1<n_2<\cdots<n_t} \frac{q^{r_1 n_1+r_2 n_2+\cdots+r_t n_t}}{(1+aq^{n_1}+q^{2n_1})^{k_1} (1+aq^{n_2}+q^{2n_2})^{k_2} \cdots (1+aq^{n_t}+q^{2n_t})^{k_t}} ,
\end{equation*}
where $t \in \IN_0$ and $r_1,\ldots,r_t,k_1,\ldots,k_t \in \IN$ (here we let $h:=1$ for the case $t=0$). Note that $\mathcal{U}_{t,k,r}(a;q)$ is equal to $h_{k,\ldots,k;r,\ldots,r}(a;q)$ with $k,r$ repeated $t$ times.
Now we consider the set $A = \IN^2$, the product on $\IQ A$ defined by linearly extending \smash{$(k_1,r_1)\diamond (k_2,r_2) = (k_1+k_2,r_1+r_2)$}, and the corresponding quasi-shuffle product on $\QA$ defined as in \eqref{eq:quasi_shuffle_product}.
Then the mapping 
\begin{equation}\label{eq:qseries_quasi_shuffle_homomorphism}
\langle A \rangle  \ni (k_1,r_1)\ldots (k_t,r_t)
\mapsto 
h_{k_1,\ldots,k_t;r_1,\ldots,r_t} (a;q)
\end{equation}
linearly extends to an algebra homomorphism from the quasi-shuffle algebra $\QA$ to $\mathcal{H}_{a,q}$ (see for example \cite[Lemma~2.18]{Bac25} for details). 
So the algebraic structure of $\QA$ has immediate implications for the multiplicative behavior of the $q$-series $h_{k_1,\ldots,k_t;r_1,\ldots,r_t}(a;q)$. For example, we have
\begin{equation*}
h_{k;r} (a;q) \, h_{\ell;s} (a;q) = 
h_{k,r;\ell,s} (a;q) + h_{\ell,s;k,r} (a;q) + h_{k+\ell;r+s} (a;q)
\end{equation*}
as can be seen by disentangling the involved sums as
\begin{multline*}
\sum_{n_1\geq 1}  \frac{q^{n_1r}}{(1+aq^{n_1}+q^{2n_1})^k} 
\sum_{n_2\geq 1} \frac{q^{n_2s}}{(1+aq^{n_2}+q^{2n_2})^\ell} 
= \!\!\sum_{1\le n_1<n_2} \!\!\frac{q^{n_1r+n_2s}}{(1+aq^{n_1}+q^{2n_1})^k(1+aq^{n_2}+q^{2n_2})^\ell} 
\\
+ \!\!\!\!\sum_{1\le n_2<n_1}  \!\!\frac{q^{n_2s+n_1r}}{(1+aq^{n_2}+q^{2n_2})^\ell (1+aq^{n_1}+q^{2n_1})^k}
+ \sum_{n\ge 1} \!\!\frac{q^{n(r+s)}}{(1+aq^{n}+q^{2n})^{k+\ell}}
\end{multline*}
and the right hand side of this identity mirrors 
$(k,r) \qsh (\ell,s) = (k,r)(\ell,s) + (\ell,s)(k,r) + (k+\ell, r+s)$.

One property of quasi-shuffle algebras that would be key to our work is the following identity given in \cite{HI17} as equation~(32) for any letter $a \in A$:
\begin{equation}\label{eq:qsh_identity}
\exp_{\qsh}\!\left(\sum_{n\geq 1} \frac{(-1)^{n+1}}{n}  a^{\diamond n} X^n \right) = \sum_{j\geq 0} a^{\circ j} X^j.
\end{equation}
Here $a^{\diamond n}$ and $a^{\circ n}$ are the $n$-fold diamond product and concatenation, respectively, of the letter $a \in A$ and
\begin{equation*}
\exp_{\qsh}(uX) := \sum_{n\geq 0} u^{\qsh n} \frac{ X^n}{n!} 
\quad \mbox{for } u \in \QA [[X]]
\end{equation*}
with $u^{\qsh n}$ denoting the $n$-fold quasi-shuffle product.

%%%%%%%%%%%%%%%%%%%
\section{quasimodularity of $\mathcal{U}_{t,k,k}$}\label{sec:quasimod_Utkk}
In this section, our goal is to prove Theorem~\ref{thm:weight-and-level-of-U} along with generalizations in select parts.
We start with part (2) and discuss the decomposition of $\mathcal{U}_{t,k,k}(a;q)$ to $\mathcal{U}_{k,k}(a;q), \mathcal{U}_{2k,2k}(a;q), \ldots, \mathcal{U}_{tk,tk}(a;q)$. For this purpose, we use the following key lemma deduced from the identity~\ref{eq:qsh_identity}.

\begin{lemma}\label{lem:QS-U}
	Let $k,r\in\N$. Then we have
	\begin{align*}
		\exp\lrb{\sum_{n\ge 1}\frac{(-1)^{n+1}}{n}\mathcal{U}_{nk,nr}(a;q)X^n}= 1+ \sum_{j\geq 1} \mathcal{U}_{j,k,r}(a;q) X^j.
	\end{align*}
\end{lemma}
\begin{proof}
Consider the quasi-shuffle algebra on $A = \IN^2$ from Section \ref{sec:quasi_shuffle_intro}. We use the identity \eqref{eq:qsh_identity} with the letter $a=(k,r)$ and then use the homomorphism \eqref{eq:qseries_quasi_shuffle_homomorphism} to turn that into an identity on $q$-series (since the quasi-shuffle product corresponds to the ordinary multiplication of $q$-series under this homomorphism).
The lemma statement follows once we note that \eqref{eq:qseries_quasi_shuffle_homomorphism} maps $a^{\diamond n} = (nk,nr)$ to $h_{nk;nr} (a;q) = \mathcal{U}_{nk,nr}(a;q)$
and $a^{\circ j} = (k,r)\ldots (k,r)$ with $k,r$ repeated $j$ times to $h_{k,\ldots,k;r,\ldots,r} (a;q) = \mathcal{U}_{j,k,r}(a;q)$ for $j \geq 1$. 
\end{proof}

We are now ready to prove Theorem~\ref{thm:weight-and-level-of-U} (2) by giving the precise form of the aforementioned decomposition of $\mathcal{U}_{t,k,k}(a;q)$, which we state more generally for $\mathcal{U}_{t,k,r}(a;q)$.

\subsection{Proof of Theorem~\ref{thm:weight-and-level-of-U} (2)}
We use the generating function of P\'olya cycle index polynomial for symmetric groups (see Example 5.2.10 of \cite{Stanley}) to rewrite the left hand side of Lemma~\ref{lem:QS-U} as
\begin{align*}
\exp\!\lrb{\sum_{n\ge 1}\frac{(-1)^{n+1}}{n}\mathcal{U}_{nk,nr}(a;q)X^n} =
1 + \sum_{n\ge 1} \sum_{\lambda\vdash n} \prod_{s=1}^n \frac{1}{m_{\l,s}!}\lrb{\frac{(-1)^{s+1}\, \mathcal{U}_{sk,sr}(a;q)}{s}}^{m_{\l,s}} X^n,
\end{align*}
where $m_{\l,s}$ denotes the multiplicity of $s$ in $\lambda$.
Comparing this with the right hand side of Lemma~\ref{lem:QS-U} yields
\begin{align*}
	\mathcal{U}_{t,k,r}(a;q) = \sum_{\lambda\vdash t} \prod_{s=1}^t \frac{1}{m_{\l,s}!}\lrb{\frac{(-1)^{s+1}\mathcal{U}_{sk,sr}(a;q)}{s}}^{m_{\l,s}}.
\end{align*}
Note that the result is an isobaric polynomial of degree $t$ in the variables 
$\mathcal{U}_{k,r}(a;q), \mathcal{U}_{2k,2r}(a;q), \ldots, \mathcal{U}_{tk,tr}(a;q)$
since $\sum_{s=1}^n s \, m_{\l,s}=t$. The argument for the level is clear from part (1).  \qed

\subsection{Proof of Theorem \ref{thm:weight-and-level-of-U} (1)}
Our next goal is proving the quasimodularity of the functions $\mathcal{U}_{k,k}(a;q)$ appearing in the decomposition of $\mathcal{U}_{t,k,k}(a;q)$. We treat the cases where $a$ equals $0,2,1,-1$, in succession.

\subsection{The Case $a=0$} We give the following precise version for the quasimodularity of $\mathcal{U}_{k,k}(0;q)$. Here $\chi_{4,2}$ denotes the unique primitive character modulo $4$ given by
\begin{align*}
	\chi_{4,2}(m) = \begin{cases} 
						0 & \text{ if }  m\equiv 0\Pmod{2},\\
						1 & \text{ if } m\equiv 1\Pmod{4},\\
						-1 & \text{ if } m\equiv 3\Pmod{4}.
					\end{cases}
\end{align*}

\begin{theorem}\label{thm:a=0-case}
For $k\in\N$, we have
\begin{align*}
		\mathcal{U}_{k,k}(0;q) &= 
		- 2^{-k-1}  + 		
		\begin{cases}
									\displaystyle\sum_{j=1}^{\frac{k}{2}} a_k(j) \lrb{2^{2j} G_{2j}(4\tau)- G_{2j}(2\tau)}  & \text{if } 2 \mid k ,
									\vspace{1ex} \\ 
									\displaystyle\sum_{j=0}^{\frac{k-1}{2}} b_k(j) G_{2j+1}(\chi_{4,2};\tau)  & \text{if } 2 \nmid k,
								\end{cases}
\end{align*}
with the coefficients $a_k (j)$ and $b_k (j)$ defined through\footnote{
More explicitly we have $\smash{a_k\!\lrb{\frac{k}{2}}=(-1)^{\frac{k}{2}}\!\frac{1}{(k-1)!},\, b_{k}\!\lrb{\frac{k-1}{2}}=(-1)^{\frac{k-1}{2}}\!\frac{1}{2^{k-1}(k-1)!}}$ and for $j< \lfloor \frac{k}{2} \rfloor$
\begin{align*}
a_k(j) = \frac{(-1)^{j}}{(k-1)!} \!\!
\sum_{\substack{\ell_s\in\left\{1,2,\ldots,\frac{k}{2}-1\right\}\\ \ell_1<\ell_2<\ldots <\ell_{\frac{k}{2}-j}}} \!\!\!\!\!\!\!\!\!\! (\ell_1\ell_2\cdots \ell_{\frac{k}{2}-j})^2
\andd
b_k(j) = \frac{(-1)^{j}}{2^{k-1}(k-1)!} \!\! \sum_{\substack{\ell_s\in\left\{1,2,\ldots,k-2\right\}\\ \ell_1< \ell_2 <\cdots <\ell_{\frac{k-1}{2}-j} \\ \ell_s: \,\mathrm{odd}}} \!\!\!\!\!\!\!\!\!\! \left( \ell_1 \ell_2 \cdots \ell_{\frac{k-1}{2}-j} \right)^2 .
\end{align*}
}
\begin{equation}\label{eq:a0_case_ak_bk_coefs}
\sum_{j=1}^{\frac{k}{2}} a_k(j) \, m^{2j-1}
:= (-1)^{\frac{k}{2}}\binom{m+\frac{k}{2}-1}{k-1}
\andd
\sum_{j=0}^{\frac{k-1}{2}} b_k(j) \, m^{2j} 
:= 
(-1)^{\frac{k-1}{2}}\binom{\frac{m+k}{2}-1}{k-1} .
\end{equation}
In particular, $\mathcal{U}_{k,k}(0;q)$ is quasimodular (modular) of level $4$ and highest weight~$k$ for $k$ even (odd).
\end{theorem}
\begin{proof}
Using the binomial expansion of the denominator we have
\begin{equation}\label{eq:U-as-der}
\mathcal{U}_{k,k}(0;q) = \sum_{n\geq 1} \frac{q^{kn}}{(1+q^{2n})^k}
= 
\sum_{\substack{n \geq 1 \\ m \geq 0}} (-1)^m \binom{m+k-1}{k-1} q^{n(2m+k)} .
\end{equation}
We now consider the even and odd $k$ cases separately. If $k\in2\IN$, then we change variables as \smash{$m \mapsto m-\frac{k}{2}$} and extend the resulting sum over $m \geq \frac{k}{2}$ to one over $m \geq 1$ noticing that the binomial coefficient vanishes at the new points. This yields
\begin{equation*}
\mathcal{U}_{k,k}(0;q) 
=  \sum_{n,m \geq 1} 
(-1)^{m-\frac{k}{2}} \binom{m+\frac{k}{2}-1}{k-1} q^{2mn}.
\end{equation*}
We note that \smash{$m \mapsto (-1)^{\frac{k}{2}}\binom{m+\frac{k}{2}-1}{k-1}$} is an odd polynomial in $m$ of degree $k-1$ and its coefficients are given by $a_k (j)$ as in \eqref{eq:a0_case_ak_bk_coefs} so that 
\begin{equation*}
	\mathcal{U}_{k,k}(0;q) = \sum_{j=1}^{\frac{k}{2}} a_k(j)  \sum_{n,m \geq 1}   (-1)^{m} m^{2j-1} q^{2nm}.
\end{equation*}
Since $(-1)^m = 2\d_{2 \mid m} - 1$, the sum over $n,m$ evaluates to 
$2^{2j} F_{2j} (4 \t) - F_{2j} (2 \t)$ by equation \eqref{eq:Fk_definition}. 
Then restoring the constant Fourier coefficients we find
\begin{equation*}
\mathcal{U}_{k,k}(0;q) = 
\sum_{j=1}^{\frac{k}{2}} a_k(j) \lrb{2^{2j} G_{2j}(4\tau)- G_{2j}(2\tau)} +  \mathscr{A}_k,
\where
\mathscr{A}_k := -\frac{1}{2} \sum_{j=1}^{\frac{k}{2}}a_k(j)\!\lrb{2^{2j}-1\!}\z(1-2j) .
\end{equation*}
To study the constant term, we multiply both sides of the first equation in \eqref{eq:a0_case_ak_bk_coefs} with $z^m$ and sum both sides over $m \in \IN$ for $|z| < 1$ to obtain (using equation \eqref{eq:Eulerian_polynomial_definition})
\begin{equation*}
\sum_{j=1}^{\frac{k}{2}} a_k(j) \frac{z P_{2j-1} (z)}{(1-z)^{2j}}
=
(-1)^{\frac{k}{2}} \frac{z^{\frac{k}{2}}}{(1-z)^k} .
\end{equation*}
Evaluating this identity at $z=-1$ using \eqref{eq:Eulerian_polynomial_zeta_relation} then yields that $\mathscr{A}_k = - 2^{-k-1}$. The theorem statement then follows by noting that the resulting combination $2^{2j} G_{2j} (4 \t) - G_{2j} (2 \t)$ is a modular form of weight $2j$ for $j \geq 2$ and a quasimodular form of weight $2$ for $j=1$ on the congruence group $\Gamma_0 (4)$ (with trivial character).

Similarly for $k \in 2 \IN -1$ we change variables as \smash{$m \mapsto \frac{m-k}{2}$} and extend the resulting sum over odd $m$ with $m \geq k$ to one over $m \geq 1$ again noticing that the binomial coefficient vanishes at these new points:
\begin{equation*}
\mathcal{U}_{k,k}(0;q) = \!\!\!\!
\sum_{\substack{n,m\ge 1\\m\equiv 1\hspace{-0.2cm}\pmod{2}}} \!\!\!\!\!\! (-1)^{\frac{m-k}{2}} \binom{\frac{m+k}{2}-1}{k-1} q^{nm}. 
\end{equation*}
Note that \smash{$m \mapsto (-1)^{\frac{k-1}{2}}\binom{\frac{m+k}{2}-1}{k-1}$} is an even polynomial in $m$ of degree $k-1$ and its coefficients are given by $b_k (j)$ as in \eqref{eq:a0_case_ak_bk_coefs}. Also note that \smash{$\chi_{4,2} (m) = (-1)^{\frac{m-1}{2}}$} for odd $m$ (while vanishing for even $m$) to write
\begin{equation*}
	\mathcal{U}_{k,k}(0;q) = \sum_{j=0}^{\frac{k-1}{2}} b_k(j) \sum_{n,m \geq 1} \chi_{4,2}(m) \, m^{2j} q^{mn} .
\end{equation*}
The sum over $n,m$ evaluates to $F_{2j+1} (\chi_{4,2}; \t)$ by  \eqref{eq:Fk_definition} and restoring the constant Fourier coefficient yields
\begin{equation*}
\mathcal{U}_{k,k}(0;q) = 
\sum_{j=0}^{\frac{k-1}{2}} b_k(j) G_{2j+1}(\chi_{4,2};\tau) + \mathscr{B}_k,
\where
\mathscr{B}_k:=- \sum_{j=0}^{\frac{k-1}{2}} b_k(j) \frac{L(\chi_{4,2},-2j)}{2}.
\end{equation*}
As above (using \eqref{eq:Eulerian_polynomial_definition}), we multiply both sides of the second equation in \eqref{eq:a0_case_ak_bk_coefs} with $z^m$ and sum both sides over $m \in 2\IN-1$ for $|z| < 1$ to obtain 
\begin{equation*}
\sum_{j=0}^{\frac{k-1}{2}} b_k(j) 
\lp \frac{z P_{2j} (z)}{(1-z)^{2j+1}} - 2^{2j} \frac{z^2 P_{2j} (z^2)}{(1-z^2)^{2j+1}} \rp
=
(-1)^{\frac{k-1}{2}} \frac{z^{k}}{(1-z^2)^k} .
\end{equation*}
Evaluating at $z=i$ using \eqref{eq:Eulerian_polynomial_zeta_relation} then yields that $\mathscr{B}_k = - 2^{-k-1}$ as well.
The theorem statement then follows by noting that $G_{2j+1} (\chi_{4,2}; \t)$ belongs to $M_{2j+1}(\Gamma_0(4),\chi_{4,2})$ for all $j \geq 0$.
\end{proof}

\subsection{The Case $a=2$} We again give the precise decomposition of $\mathcal{U}_{k,k}(2;q)$ to Eisenstein series.
\begin{theorem}
For $k\in\N$, we have (with $a_\ell(j)$ defined as in Theorem~\ref{thm:a=0-case})
\begin{align*}
	\mathcal{U}_{k,k}(2;q) &= -2^{-2k-1} + \sum_{j=1}^{k} a_{2k}(j) \lrb{ 2^{2j}G_{2j}(2\tau) - G_{2j}(\tau)} .
\end{align*}
In particular, $\mathcal{U}_{k,k}(2;q)$ is quasimodular of level $2$ and highest weight $2k$.
\end{theorem}
\begin{proof}
The result immediately follows from Theorem~\ref{thm:a=0-case} once we note
$\mathcal{U}_{k,k}(2;q) = \mathcal{U}_{2k,2k}(0;q^{\frac{1}{2}})$.
\end{proof}

\subsection{The Case $a=1$}
Here we state a more general result that allows us to prove the quasimodularity of the combinations $\mathcal{U}_{k,r}(1;q)+\mathcal{U}_{k,2k-r}(1;q)$ as well. 
In the theorem statement below, note that  $\mathcal{U}_{k,k}(1;q)$ corresponds to choosing \smash{$Q_k(x)=x^k$}. 
Also as in Example~\ref{eg:1}, we use
$\chi_{3,2}$ to denote the unique primitive character modulo $3$ given by
\begin{align*}
	\chi_{3,2}(m) = \begin{cases} 
						0 & \text{ if }  m\equiv 0\Pmod{3},\\
						1 & \text{ if } m\equiv 1\Pmod{3},\\
						-1 & \text{ if } m\equiv 2\Pmod{3}.
					\end{cases}
\end{align*}

\begin{theorem}\label{thm:Qk_quasipolynomial_a_1}
Let $k \in \IN$ and $Q_k$ be a polynomial of degree $\leq 2k-1$ satisfying $Q_k (x) = x^{2k} Q_k(1/x)$ and such that $Q_k (0) = 0$. Then we have
\begin{equation}\label{eq:Qk_quasipolynomial_a_1}
\sum_{n \geq 1} \frac{Q_k (q^n)}{(1+q^n+q^{2n})^k}
= 
- \frac{Q_k (1)}{2 \cdot 3^k}
+
i \sqrt{3} \sum_{\substack{\ell=1 \\ \ell: \, \mathrm{odd}}}^k 
c_{Q_k}\!(\ell) \, G_\ell (\chi_{3,2}; \t) 
+
\sum_{\substack{\ell=1 \\ \ell: \, \mathrm{even}}}^k 
c_{Q_k}\!(\ell) \! \lp 3^\ell G_\ell (3 \t)  - G_\ell (\t) \rp ,
\end{equation}
where
\begin{equation*}
c_{Q_k}\!(\ell) \! := \!\sum_{r=\ell}^k \frac{a_{Q_k}\!(r)}{(r-1)!} \St{r-1 \\ \ell -1}
\end{equation*}
with $a_{Q_k}\!(r)$ coming from the decomposition 
%of $\frac{Q_k(x)/x}{(1+x+x^2)^k}$ as
\begin{equation*}
\frac{Q_k(x)}{(1+x+x^2)^k} =
\sum_{r=1}^k a_{Q_k}\!(r) \frac{\z_3 x}{(1-\z_3 x)^r}
+ \sum_{r=1}^k a'_{Q_k}\!(r) \frac{\z^{-1}_3 x}{(1-\z^{-1}_3 x)^r} .
\end{equation*}
In particular, \eqref{eq:Qk_quasipolynomial_a_1} is a linear combination of modular forms (and quasimodular for weight $2$) with level $3$ and highest weight $k$.
\end{theorem}
\begin{proof}
Starting with the decomposition of $\frac{Q_k(x)}{(1+x+x^2)^k}$ in the theorem statement and applying \eqref{eq:eulerian_decomposition}, we find
\begin{equation*}
\frac{Q_k(x)}{(1+x+x^2)^k}
=
\sum_{r=1}^k \frac{a_{Q_k}\!(r)}{(r-1)!} \sum_{\ell = 1}^r \St{r-1 \\ \ell-1} 
\frac{\z_3 x \, P_{\ell-1} (\z_3 x)}{(1-\z_3 x)^\ell}
+ \sum_{r=1}^k \frac{a'_{Q_k}\!(r)}{(r-1)!} \sum_{\ell = 1}^r \St{r-1 \\ \ell-1} 
\frac{\z^{-1}_3 x \, P_{\ell-1} (\z^{-1}_3 x)}{(1-\z^{-1}_3 x)^\ell}.
\end{equation*}
Changing the order of the summation, we rewrite this as
\begin{equation}\label{eq:Qk_partial_fraction_eulerian_decomposition}
\frac{Q_k(x)}{(1+x+x^2)^k}
=
\sum_{\ell = 1}^k c_{Q_k}\!(\ell)
\frac{\z_3 x \, P_{\ell-1} (\z_3 x)}{(1-\z_3 x)^\ell}
+ \sum_{\ell = 1}^k c'_{Q_k}\!(\ell)
\frac{\z^{-1}_3 x \, P_{\ell-1} (\z^{-1}_3 x)}{(1-\z^{-1}_3 x)^\ell},
\end{equation}
where $c'_{Q_k}\!(\ell)$ is defined similar to $c_{Q_k}\!(\ell)$ with $a_{Q_k}\!(r)$ replaced by $a'_{Q_k}\!(r)$. 
Now changing $x \mapsto 1/x$ leaves the left hand side invariant thanks to our assumption on $Q_k$ and hence, we find
\begin{equation*}
\frac{Q_k(x)}{(1+x+x^2)^k}
=
\sum_{\ell = 1}^k (-1)^\ell c_{Q_k}\!(\ell)
\frac{(\z_3^{-1} x)^{\ell-1} \, P_{\ell-1} \!\lp \frac{1}{\z_3^{-1} x} \rp}{(1-\z_3^{-1} x)^\ell}
+ \sum_{\ell = 1}^k (-1)^\ell c'_{Q_k}\!(\ell)
\frac{(\z_3 x)^{\ell-1} \, P_{\ell-1} \!\lp \frac{1}{\z_3 x} \rp}{(1-\z_3 x)^\ell} .
\end{equation*}
For $\ell \geq 2$, we can use the symmetry relation \eqref{eq:Eulerian_polynomial_symmetry} to write
\begin{equation*}
\lp \z_3^{\pm 1} x \rp^{\ell -2} P_{\ell-1} \!\lp \frac{1}{\z_3^{\pm 1} x} \rp
= P_{\ell-1} (\z_3^{\pm 1} x) .
\end{equation*}
For $\ell = 1$, on the other hand, we have $P_0 (t) = 1$ and hence 
\begin{equation*}
\frac{(\z_3^{\pm 1} x)^{\ell-1} \, P_{\ell-1} \!\lp \frac{1}{\z_3^{\pm 1} x} \rp}{(1-\z_3^{\pm 1} x)^\ell}
=
\frac{1}{1-\z_3^{\pm 1} x} 
= 1 + \frac{\z^{\pm 1}_3 x \, P_{\ell-1} (\z^{\pm 1}_3 x)}{(1-\z^{\pm 1}_3 x)^\ell}.
\end{equation*}
Thus overall we find
\begin{equation*}
\frac{Q_k(x)}{(1+x+x^2)^k}
=
-(c_{Q_k}\!(1) +c'_{Q_k}\!(1))
+
\sum_{\ell = 1}^k (-1)^\ell c'_{Q_k}\!(\ell)
\frac{\z_3 x \, P_{\ell-1} (\z_3 x)}{(1-\z_3 x)^\ell}
+
\sum_{\ell = 1}^k (-1)^\ell c_{Q_k}\!(\ell)
\frac{\z^{-1}_3 x \, P_{\ell-1} (\z^{-1}_3 x)}{(1-\z^{-1}_3 x)^\ell} .
\end{equation*}
The decomposition of the left hand side to such fractions is unique\footnote{
This can be seen e.g.~ by multiplying with $(1-\z_3^{\pm 1} x)^n$ and setting $x = \z_3^{\mp 1}$ for $n$ equal to $k,k-1,\ldots,1$ successively and noting that $P_{\ell-1} (1) = (\ell-1)!$ by equation \eqref{eq:Eulerian_polynomial_tm1_decomposition}.} 
and hence comparing with \eqref{eq:Qk_partial_fraction_eulerian_decomposition} we conclude that
\begin{equation*}
c'_{Q_k}\!(\ell) = (-1)^\ell \, c_{Q_k}\!(\ell) \quad \mbox{for all } \ell \in \{1,\ldots,k\}.
\end{equation*}
Therefore, we have
\begin{multline}\label{eq:Qk_partial_fraction_eulerian_decomposition_symmetric}
\frac{Q_k(x)}{(1+x+x^2)^k}
=
\sum_{\substack{\ell=1 \\ \ell: \, \mathrm{odd}}}^k c_{Q_k}\!(\ell)
\lp \frac{\z_3 x \, P_{\ell-1} (\z_3 x)}{(1-\z_3 x)^\ell}
- \frac{\z^{-1}_3 x \, P_{\ell-1} (\z^{-1}_3 x)}{(1-\z^{-1}_3 x)^\ell} \rp
\\
+
\sum_{\substack{\ell=1 \\ \ell: \, \mathrm{even}}}^k c_{Q_k}\!(\ell)
\lp \frac{\z_3 x \, P_{\ell-1} (\z_3 x)}{(1-\z_3 x)^\ell}
+ \frac{\z^{-1}_3 x \, P_{\ell-1} (\z^{-1}_3 x)}{(1-\z^{-1}_3 x)^\ell} \rp .
\end{multline}
Setting $x=q^n$ and summing over $n \in \IN$, we can use equation \eqref{eq:Eulerian_polynomial_definition} to evaluate
\begin{equation*}
\sum_{n \geq 1} \frac{Q_k (q^n)}{(1+q^n+q^{2n})^k}
=
\sum_{\substack{\ell=1 \\ \ell: \, \mathrm{odd}}}^k c_{Q_k}\!(\ell) \!
\sum_{n,m \geq 1} (\z_3^m - \z_3^{-m}) m^{\ell -1} q^{mn}
+
\sum_{\substack{\ell=1 \\ \ell: \, \mathrm{even}}}^k c_{Q_k}\!(\ell) \!
\sum_{n,m \geq 1} (\z_3^m + \z_3^{-m}) m^{\ell -1} q^{mn} .
\end{equation*}
Noting the identities
\begin{equation}\label{eq:zeta3_decomposition_characters}
\z_3^m - \z_3^{-m} = i \sqrt{3} \chi_{3,2} (m)
\andd
\z_3^m + \z_3^{-m} = 3 \d_{3 \mid m} -1 ,
\end{equation}
the sum over $n,m$ can be evaluated to $i \sqrt{3}  F_\ell (\chi_{3,2}; \t)$ for the first sum and to $3^\ell F_\ell (3 \t) - F_\ell (\t)$ for the second sum thanks to \eqref{eq:Fk_definition}.
Restoring the constant Fourier coefficients yields
\begin{equation*}
\sum_{n \geq 1} \frac{Q_k (q^n)}{(1+q^n+q^{2n})^k}
= 
i \sqrt{3} \sum_{\substack{\ell=1 \\ \ell: \, \mathrm{odd}}}^k 
c_{Q_k}\!(\ell) \, G_\ell (\chi_{3,2}; \t) 
+
\sum_{\substack{\ell=1 \\ \ell: \, \mathrm{even}}}^k 
c_{Q_k}\!(\ell) \! \lp 3^\ell G_\ell (3 \t)  - G_\ell (\t) \rp 
+
\CC_{Q_k} ,
\end{equation*}
where
\begin{equation*}
\CC_{Q_k} \! := \! - \frac{i \sqrt{3}}{2} \!
\sum_{\substack{\ell=1 \\ \ell: \, \mathrm{odd}}}^k \! c_{Q_k}\!(\ell) \, L(\chi_{3,2},1-\ell)
- \frac{1}{2} \! \sum_{\substack{\ell=1 \\ \ell: \, \mathrm{even}}}^k  \!
c_{Q_k}\!(\ell) \lp 3^\ell - 1 \rp \z (1-\ell)  . 
\end{equation*}
We find $\CC_{Q_k} = - \frac{Q_k (1)}{2 \cdot 3^k}$ by evaluating \eqref{eq:Qk_partial_fraction_eulerian_decomposition_symmetric} at $x=1$ and using \eqref{eq:Eulerian_polynomial_zeta_relation} with $r=3$, $n \in \{1,2\}$ and noting~\eqref{eq:zeta3_decomposition_characters}. 
The theorem statement then follows by noting that $G_\ell (\chi_{3,2}; \t)$ is in $M_{\ell}(\Gamma_0(3),\chi_{3,2})$ for $\ell \geq 1$ odd and $3^\ell G_\ell (3 \t) - G_\ell (\t)$ is a weight $\ell$ modular form for $\ell > 2$ even (and quasimodular for $\ell=2$) on $\Gamma_0 (3)$ with trivial character.
\end{proof}

\begin{remarks}\hfill\break
(1) Note that \smash{$c_{Q_k}\!(k) = \frac{1}{(k-1)!} a_{Q_k}\!(k)$} according to the definition of $c_{Q_k}\!(\ell)$'s and \eqref{eq:unsigned_stirling_first_kind_particular_values}. We then compute
\begin{equation*}
a_{Q_k}\!(k) 
= \lim_{x \to \z_3^{-1}} \lb (1-\z_3 x)^k \frac{Q_k(x)}{(1+x+x^2)^k} \rb 
= \frac{Q_k (\z_3^{-1})}{(1-\z_3^{-2})^k} 
\ \ \mbox{ and hence }
c_{Q_k}\!(k) = \frac{\z_3^k Q_k (\z_3^{-1})}{(k-1)! (i \sqrt{3})^k}.
\end{equation*}
Therefore, the contribution of the weight $k$ piece in Theorem \ref{thm:Qk_quasipolynomial_a_1} is nontrivial if and only if $Q_k (\z_3^{-1}) \neq 0$. 
\\
(2) Thanks to \eqref{eq:unsigned_stirling_first_kind_particular_values}, 
the coefficient of the only quasimodular contribution to \eqref{eq:Qk_quasipolynomial_a_1} is 
\begin{equation*}
c_{Q_k}\!(2) = \sum_{r=2}^k \frac{a_{Q_k}\!(r)}{r-1} \  \mathrm{for}\  k \geq 2.
\end{equation*}
\end{remarks}

\subsection{The Case $a=-1$} Just like the $a=1$ case, we give a more general result that allows us to prove the quasimodularity of the combinations $\mathcal{U}_{k,r}(1;q)+\mathcal{U}_{k,2k-r}(1;q)$ with the case of $\mathcal{U}_{k,k}(-1;q)$ given by setting \smash{$Q_k(x)=x^k$} in the theorem statement.

\begin{theorem}\label{thm:Qk_quasipolynomial_a_m1}
Let $k \in \IN$ and $Q_k$ be a polynomial of degree $\leq 2k-1$ satisfying $Q_k (x) = x^{2k} Q_k(1/x)$ and such that $Q_k (0) = 0$. Then we have
\begin{multline}\label{eq:Qk_quasipolynomial_a_m1}
\sum_{n \geq 1} \frac{Q_k (q^n)}{(1-q^n+q^{2n})^k}
= 
- \frac{Q_k (1)}{2}
+
i \sqrt{3} \sum_{\substack{\ell=1 \\ \ell: \, \mathrm{odd}}}^k 
d_{Q_k}\!(\ell) \, 
\lp 2^\ell G_\ell (\chi_{3,2}; 2\t)  + G_\ell (\chi_{3,2}; \t) \rp
\\
+
\sum_{\substack{\ell=1 \\ \ell: \, \mathrm{even}}}^k 
d_{Q_k}\!(\ell) \,  
\lp 6^\ell G_\ell (6 \t)- 3^\ell G_\ell (3 \t) 
- 2^\ell G_\ell (2 \t) + G_\ell (\t) \rp ,
\end{multline}
where
\begin{equation*}
d_{Q_k}\!(\ell) \,  := \sum_{r=\ell}^k \frac{b_{Q_k}\!(r)}{(r-1)!} \St{r-1 \\ \ell -1}
\end{equation*}
with $b_{Q_k}\!(r)$ coming from the decomposition 
%of $\frac{Q_k(x)/x}{(1-x+x^2)^k}$ as
\begin{equation*}
\frac{Q_k(x)}{(1-x+x^2)^k} =
\sum_{r=1}^k b_{Q_k}\!(r) \frac{\z_6 x}{(1-\z_6 x)^r}
+ \sum_{r=1}^k b'_{Q_k}\!(r) \frac{\z^{-1}_6 x}{(1-\z^{-1}_6 x)^r} .
\end{equation*}
In particular, \eqref{eq:Qk_quasipolynomial_a_m1} is a linear combination of modular forms (and quasimodular for weight $2$) with level $6$ and highest weight $k$.
\end{theorem}
\begin{proof}
We follow the same strategy as in the proof of Theorem \ref{thm:Qk_quasipolynomial_a_1}. The crucial point here is the decomposition of the function $m \mapsto \z_N^m \pm \z_N^{-m}$ for $m \in \IZ$ to a linear combination of $m \mapsto \d_{N_1 \mid m} \, \chi (m/N_1)$, where $\chi$ is a primitive Dirichlet character modulo $N_2$ with $N_1 N_2 \mid N$. In the case $N=6$ relevant here, this reads
\begin{equation*}
\z_6^m - \z_6^{-m} = i \sqrt{3} \lp  2 \d_{2 \mid m} \, \chi_{3,2} (m/2) + \chi_{3,2} (m) \rp 
\andd
\z_6^m + \z_6^{-m} = 6 \d_{6 \mid m} - 3 \d_{3 \mid m} - 2 \d_{2 \mid m}  + 1 .
\end{equation*}
The rest of the proof is the same as that of Theorem \ref{thm:Qk_quasipolynomial_a_1}.
\end{proof}
\begin{remark}
We have
\begin{equation*}
d_{Q_k}\!(k) = \frac{\z_6^k Q_k (\z_6^{-1})}{(k-1)! (i \sqrt{3})^k},
\end{equation*}
so the contribution of the weight $k$ piece in Theorem \ref{thm:Qk_quasipolynomial_a_m1} is nontrivial if and only if $Q_k (\z_6^{-1}) \neq 0$.
\end{remark}

\section{Limiting behavior of $\mathcal{U}_{t,k,r}$}\label{sec:limiting_Utkr}
In this section, we prove Theorem \ref{thm:lim-behaviour} by establishing it as a special case of a more general result. 
To state it, we define for $t \in \IN$, $k \in \IZ$ and polynomials
$P(x) \in x \, \N_0[x]$ of degree $m \in \N$ and $Q(x) \in x \, \Z[x]$ of degree $s \in \N$ the $q$-series
\begin{equation*}
\mathcal{U}_{t,k}(P,Q;q):= \sum_{1\le n_1<n_2<\cdots<n_t} \frac{q^{P(n_1)+P(n_2)+\cdots+P(n_t)}}{(1+Q(q^{n_1}))^k(1+Q(q^{n_2}))^k\cdots (1+Q(q^{n_t}))^k}.
\end{equation*} 

\begin{theorem}\label{General_limit}
Assuming the notations above, we have
\begin{align*}
q^{-\sum_{j=1}^t \! P(j)}\, \mathcal{U}_{t,k}(P,Q;q)
=\begin{cases}  
\prod_{n\geq 1}\frac{1}{(1-q^{\alpha n})(1+Q(q^n))^k}+O(q^{t+1}) \quad &{\text {\rm if $m= 1$ and $P(x) = \a x$}},\\\\
\prod_{n\geq 1}\frac{1}{(1+Q(q^n))^k}+O(q^{t+1}) \quad &{\text {\rm if $m\geq 2$}}.
\end{cases}
\end{align*}
\end{theorem}
\begin{proof}
We first prove the case $m=1$ and $P(x)=\alpha x$. We have
(changing $n_j \mapsto n_j + j$ in $\mathcal{U}_{t,k}(P,Q;q)$)
\begin{align*}
		\mathfrak{U}_{t,k}(P,Q;q)&:=q^{-\alpha\frac{t(t+1)}{2}}
		\, \mathcal{U}_{t,k}(P,Q;q)\prod_{n=1}^\infty(1-q^{\alpha n})(1+Q(q^n))^k\\
		&= \prod_{n=1}^\infty(1-q^{\alpha n})
		\sum_{0\le n_1\leq n_2\leq\cdots\leq n_t}
		\frac{q^{\alpha(n_1+n_2+\cdots+n_t)}}{\prod_{j=1}^t(1+Q(q^{n_j+j}))^k}\prod_{r=1}^\infty(1+Q(q^r))^k.
\end{align*}
Then changing variables in the sum to $(\lambda_1,\lambda_2,\cdots,\lambda_t)=(n_t,n_{t-1},\cdots,n_1)$, we can turn it into a sum over integer partitions as
\begin{equation*}
\mathfrak{U}_{t,k}(P,Q;q) =\prod_{n=1}^\infty(1-q^{\alpha n})\sum_{v\geq 0}\sum_{\ell=0}^t\sum_{\lambda\in\mathcal{P}(v,\ell)}\frac{q^{\alpha(\lambda_1+\lambda_2+\cdots+\lambda_t)}}{\prod_{j=1}^t(1+Q(q^{\lambda_{j}+t-j+1}))^k}\prod_{r=1}^\infty(1+Q(q^r))^k,
\end{equation*}
where $\mathcal{P}(v,\ell)$ denotes the set of partitions of $v$ with $\ell$ parts.
Here we note that if $\lambda \in \mathcal{P}(v,\ell)$ and $j>\ell$, then $\lambda_j = 0$ by convention and we use this to rewrite
\begin{align*}
\mathfrak{U}_{t,k}(P,Q;q) 
&=\prod_{n=1}^\infty(1-q^{\alpha n})\sum_{v\geq 0}\sum_{\ell=0}^t\sum_{\lambda\in\mathcal{P}(v,\ell)}\frac{q^{\alpha(\lambda_1+\lambda_2+\cdots+\lambda_\ell)}}{\prod_{j=1}^\ell (1+Q(q^{\lambda_{j}+t-j+1}))^k}\prod_{r=t-\ell+1}^\infty(1+Q(q^r))^k
\\
&=
\prod_{n=1}^\infty(1-q^{\alpha n})\sum_{v\geq 0}\sum_{\ell=0}^t\sum_{\lambda\in\mathcal{P}(v,\ell)} 
q^{\alpha(\lambda_1+\lambda_2+\cdots+\lambda_\ell)}
(1+O(q^{t-\ell+1})) .
\end{align*}
Now we note that terms with $v \geq t+1$ do not contribute up to order $O(q^{t+1})$. Similarly, contributions of the error terms $q^{\alpha(\lambda_1+\lambda_2+\cdots+\lambda_\ell)} O(q^{t-\ell+1})$ are zero up to $O(q^{t+1})$ since $\alpha (\lambda_1+\lambda_2+\cdots+\lambda_\ell) \geq \ell$. So we have
\begin{equation*}
\mathfrak{U}_{t,k}(P,Q;q) 
= \prod_{n=1}^\infty(1-q^{\alpha n}) \sum_{v=0}^t \sum_{\ell=0}^t\sum_{\lambda\in\mathcal{P}(v,\ell)} q^{\a v} + O(q^{t+1})
=
\prod_{n=1}^\infty(1-q^{\alpha n}) \sum_{v=0}^t p(v) q^{\a v} + O(q^{t+1}) 
\end{equation*}
where $p(v)$ is the total number of partitions of $v$ (since a partition of $v \leq t$ has at most $t$ parts). The first term then gives $1+O(q^{t+1})$ and the case $m=1$ follows. 

In the case $m\ge 2$, we analogously consider
\begin{align*}
	\widetilde{\mathfrak{U}}_{t,k}(P,Q;q) &:=q^{-\sum_{j=1}^t \! P(j)}
		\, \mathcal{U}_{t,k}(P,Q;q)\prod_{n=1}^\infty(1+Q(q^n))^k.
\end{align*}
By exactly the same calculation as in the case $m=1$, we get
\begin{align*}
	\widetilde{\mathfrak{U}}_{t,k}(P,Q;q) &= 1 + \sum_{v\geq 1}\sum_{\ell=1}^t\sum_{\lambda\in\mathcal{P}(v,\ell)} 
q^{\sum_{j=1}^t \! \lrb{P(\lambda_{t-j+1}+j)-\! P(j)}}
(1+O(q^{t-\ell+1})) . 
\end{align*}
For any $\lambda\in\mathcal{P}(v,\ell)$ with $v\ge 1$, we have
\begin{align*}
	\sum_{j=1}^t \! \lrb{P(\lambda_{t-j+1}+j)-\! P(j)} & \ge \sum_{j=t-\ell+1}^t \lrb{P(1+j)-\! P(j)} \ge t+1,
\end{align*}
where the lower bound on the telescoping series follows from restricting to the leading term of $P$, which is a polynomial with nonnegative coefficients.
Hence, we have \smash{$\widetilde{\mathfrak{U}}_{t,k}(P,Q;q) = 1+O(q^{t+1})$},
which implies the result for $m\geq2$. 
\end{proof}
\begin{proof}[Proof of Theorem \ref{thm:lim-behaviour}]
The result follows by restricting $k$ to $\IN$ and letting $P(x)=rx$ and $Q(x)=ax+x^2$ in Theorem \ref{General_limit}.
\end{proof}

%%%%%%%%%%%%%%%%%%%%%%%%%%%%%%%%%%%%%%%%%%%%%%%%%%%%%
\section{Proof of Theorem \ref{Explicit_Identities}}\label{sec:explicit_identity}
\begin{proof}[Proof of Theorem~\ref{Explicit_Identities}]
	We first recall the Jacobi triple product identity 
	\begin{align}\label{JCT}
		\sum_{n=-\infty}^{\infty}q^{\frac{n(n+1)}{2}}z^{n}=\left(1+z^{-1}\right)(q;q)_{\infty}\prod_{n=1}^{\infty}(1+z^{-1}q^{n})(1+zq^{n}).
	\end{align}
We also recall the generating function of $\mathcal{U}_t(a;q)$ (with the convention $\mathcal{U}_0 := 1$) as (see Theorem~2.1 of \cite{AAT2})
	\begin{align}\label{Gen}
		\sum_{m=0}^\infty \mathcal{U}_m(a;q) \, x^{2m}
		=\prod_{n=1}^\infty\left(1+\frac{x^2 q^n}{1+aq^n+q^{2n}}\right).
	\end{align}
Combining \eqref{JCT}, \eqref{Gen} and letting $z+z^{-1}=x^2+a$, we find
\begin{align*}
\frac{\sum_{t=-\infty}^{\infty}q^{\frac{t(t+1)}{2}} z^t}{\prod_{n\geq 1}(1-q^n)(1+aq^n+q^{2n})}
=\left(1+z^{-1}\right)\sum_{m=0}^\infty \mathcal{U}_m(a;q)~(z+z^{-1}-a)^{m}.
\end{align*}
For the theta function on the left hand side we let $t \mapsto -t-1$ for the $t<0$ terms and on the right hand side we perform the binomial expansion for \smash{$(z+z^{-1}-a)^{m}$} (with variables \smash{$z+z^{-1}$} and $-a$) to obtain
\begin{equation*}
\frac{\sum_{t=0}^{\infty}q^{\frac{t(t+1)}{2}} \lp z^t + z^{-t-1} \rp}{\prod_{n\geq 1}(1-q^n)(1+aq^n+q^{2n})}
=(1+z^{-1})
\sum_{m=0}^\infty \mathcal{U}_m(a;q) \sum_{\g=0}^m
\binom{m}{\g} \lp z+z^{-1} \rp^\g (-a)^{m-\g} .
\end{equation*}
Then we use the identity 
\begin{equation*}
(1+z^{-1}) \lp z+z^{-1} \rp^\g 
= \sum_{\a=0}^\g \binom{\g}{\lfloor \frac{\g-\a}{2} \rfloor} 
\lp z^\a + z^{-\a-1} \rp 
\end{equation*}
and switch the order of the sum over $\a$ with the sums over $m$ and $\g$ to get
\begin{equation*}
\frac{\sum_{t=0}^{\infty}q^{\frac{t(t+1)}{2}} \lp z^t + z^{-t-1} \rp}{\prod_{n\geq 1}(1-q^n)(1+aq^n+q^{2n})}
=
\sum_{\a=0}^\infty \lp z^\a + z^{-\a-1} \rp 
\sum_{m=\a}^\infty \mathcal{U}_m(a;q) \sum_{\g=\a}^m \binom{m}{\g}
\binom{\g}{\lfloor \frac{\g-\a}{2} \rfloor}  (-a)^{m-\g} .
\end{equation*}
The theorem statement follows by comparing the coefficients of $z$ on both sides and letting $\g \mapsto m-\g$ for the sum over $\g$.
\end{proof}

\end{document}